\documentclass[12pt,reqno]{amsart}
\usepackage{amssymb}
\usepackage{graphicx}
\usepackage{xcolor}

\usepackage[all]{xy}
\DeclareFontFamily{U}{mathb}{\hyphenchar\font45}
\DeclareFontShape{U}{mathb}{m}{n}{
      <5> <6> <7> <8> <9> <10> gen * mathb
      <10.95> mathb10 <12> <14.4> <17.28> <20.74> <24.88> mathb12
      }{}
\DeclareSymbolFont{mathb}{U}{mathb}{m}{n}
\DeclareMathSymbol{\righttoleftarrow}{3}{mathb}{"FD}

\theoremstyle{plain}
\newtheorem{prop}{Proposition}[]
\newtheorem{theo}[prop]{Theorem}

\newtheorem{lemm}[prop]{Lemma}
\theoremstyle{remark}
\newtheorem{rema}[prop]{Remark}

\theoremstyle{definition}
\newtheorem{defi}[prop]{Definition}

\newtheorem{exam}[prop]{Example}
\numberwithin{equation}{section}

\newcommand{\bP}{{\mathbb P}}

\newcommand{\cX}{{\mathcal X}}
\newcommand{\cY}{{\mathcal Y}}
\newcommand{\cZ}{{\mathcal Z}}

\newcommand{\lra}{\longrightarrow}

\newcommand{\bZ}{{\mathbb Z}}

\newcommand{\eqto}{\stackrel{\lower1.5pt\hbox{$\scriptstyle\sim\,$}}\to}
\newcommand{\eqdashto}{\stackrel{\lower1.5pt\hbox{$\scriptstyle\sim\,$}}\dashrightarrow}
\newcommand{\actsfromleft}{\mathrel{\reflectbox{$\righttoleftarrow$}}}
\newcommand{\actsfromright}{\righttoleftarrow}

\DeclareMathOperator{\Aut}{Aut}

\DeclareMathOperator{\Var}{Var}
\DeclareMathOperator{\Burn}{Burn}
\DeclareMathOperator{\Bir}{Bir}
\DeclareMathOperator{\cBir}{\mathcal B\mathrm{ir}}
\DeclareMathOperator{\Cr}{Cr}

\DeclareMathOperator{\Ex}{Ex}

\begin{document}
\title[Invariants of birational maps]{Burnside groups and orbifold invariants of birational maps}

\author{Andrew Kresch}
\address{
  Institut f\"ur Mathematik,
  Universit\"at Z\"urich,
  Winterthurerstrasse 190,
  CH-8057 Z\"urich, Switzerland
}
\email{andrew.kresch@math.uzh.ch}
\author{Yuri Tschinkel}
\address{
  Courant Institute,
  251 Mercer Street,
  New York, NY 10012, USA
}

\email{tschinkel@cims.nyu.edu}

\address{Simons Foundation\\
160 Fifth Avenue\\
New York, NY 10010\\
USA}

\date{August 11, 2022}

\begin{abstract}
We construct new invariants of equivariant birational isomorphisms taking values in equivariant Burnside groups. 
\end{abstract}

\maketitle

\section{Introduction}
\label{sec.intro}

In this note, we generalize the new \emph{Burnside invariant} of birational isomorphisms introduced and studied in \cite{LSh}, \cite{LSZ}. 

Let 
$$
\phi\colon X\dashrightarrow Y
$$ 
be a birational isomorphism of smooth projective $n$-dimensional algebraic varieties over an algebraically closed field $k$ of characteristic zero. 
The invariant takes values in the Burnside group
\[
\Burn_{n-1}=\bZ[\Bir_{n-1}], 
\]
of \cite{KT}, the free abelian group on birational equivalence classes of algebraic varieties of dimension $n-1$ over $k$, i.e., isomorphism classes of function fields of transcendence degree $n-1$ over $k$ (we drop the dependence on $k$ from the notation). We denote by $[K]$ the class of the function field $K$. 

To define the invariant, let $\mathrm{Ex}(\phi)$ and $\mathrm{Ex}(\phi^{-1})$ be the sets of divisorial components of the exceptional locus of $\phi$, respectively $\phi^{-1}$.   
The main observation of \cite{LSh} was that 
\begin{equation}
\label{eqn:c}
c(\phi) := \sum_{E\in \Ex(\phi^{-1})}
[k(E)] 
-\sum_{D\in \Ex(\phi)} [k(D)] \in \Burn_{n-1}
\end{equation}
respects composition \cite[Lemma 2.2]{LSh}:
$$
c(\psi\circ\phi) = c(\phi)+c(\psi);
$$
here
the sums are over divisors in the exceptional loci. 
Thus there is an induced homomorphism 
\begin{equation}
\label{eqn:cc}
c\colon\Bir\Aut(X) \to \Burn_{n-1}.
\end{equation}
from the group of birational automorphisms of $X$ over $k$ to the Burnside group.
There is also a \emph{motivic} refinement $\tilde c$ of $c$, taking values in the
{\em truncated Grothendieck group} of algebraic varieties, 
$$
\tilde{c}(\phi) \in \mathrm K_0(\Var^{\le n-1}),
$$
so that the homomorphism $c$  in \eqref{eqn:cc} factors as
$$
\Bir\Aut(X) \stackrel{\tilde{c}}{\to} \mathrm K_0(\Var^{\le n-1}) \to \Burn_{n-1},
$$
where the rightmost map, the projection to classes of  $(n-1)$-dimensional varieties, is surjective. 

Using these invariants Lin and Shinder obtained new structural results about Cremona groups 
$$
\Cr_n=\Bir\Aut(\bP^n), 
$$
see, e.g., \cite[Theorem 1.2]{LSh}. The nontriviality of these invariants in applications is based on the existence of
nontrivial L-equivalences of algebraic varieties.

One can ask about equivariant birational isomorphisms of $G$-varieties, i.e., varieties  which are equipped with an action of a finite group $G$, or about birational isomorphisms between orbifolds. Building on work in \cite{BnG} and \cite{Bbar}, 
we propose a generalization of \eqref{eqn:c} to the equivariant and orbifold contexts, by incorporating information about the induced group action, respectively, orbifold structure, on components of exceptional loci. In the equivariant context, the invariants take values in 
$$
\bZ[\Bir_{G,n-1}],
$$
the free abelian group on $G$-equivariant birational equivalence classes of
varieties of dimension $n-1$ over $k$ with generically free $G$-action.
There is an associated \emph{equivariant Burnside group}
\[ \Burn_{n-1}(G) \]
defined in \cite{BnG}, with a homomorphism
$$
\bZ[\Bir_{G,n-1}]\to \Burn_{n-1}(G);
$$
computations in the latter group allow to 
distinguish classes in $\bZ[\Bir_{G,n-1}]$, see
\cite{HKTsmall} and \cite{KT-struct} for examples. 

In the orbifold context, the invariants take values in 
$$
\bZ[\cBir_{n-1}],
$$
the free abelian group on birational equivalence classes of projective orbifolds of dimension $n-1$ over $k$, which, in turn, admits a homomorphism 
\[
\bZ[\cBir_{n-1}] \to \overline{\Burn}_{n-1}
\]
to the \emph{orbifold Burnside group},
defined in \cite{Bbar}.

Concretely, let $G$ be a finite group, acting generically freely on smooth projective $n$-dimensional varieties $X$, $Y$ over $k$. Let 
$$
\phi\colon X \dashrightarrow Y
$$ 
be a $G$-equivariant birational isomorphism.  
Let $\Ex_G(\phi)$ be the set of $G$-orbits of divisorial components of the exceptional locus of $\phi$.
As an analogue of \cite[Theorem 2.8]{LSh}, we obtain:

\begin{theo}
\label{thm:main}
The assignment
\begin{equation*}
\label{eqn:cg}
C_G(\phi) := \sum_{\substack{E\in \Ex_G(\phi^{-1})\\ \mathrm{gen. stab}(E)=\{1\}}}
[E\actsfromright G] 
-\sum_{\substack{D\in \Ex_G(\phi)\\ \mathrm{gen. stab}(D)=\{1\}}} 
[D\actsfromright G] \in \bZ[\Bir_{G,n-1}],
\end{equation*}
where the sums are over divisorial orbits of components of exceptional loci with trivial generic stabilizer, 
respects compositions 
$$
C_G(\psi\circ\phi) = C_G(\phi)+C_G(\psi), 
$$
for sequences of equivariant birational isomorphisms
$$
X\stackrel{\phi}{\dashrightarrow} Y\stackrel{\psi}{\dashrightarrow} Z.
$$
\end{theo}

Passing to the orbifold context, let 
$$
\phi\colon \mathcal X\dashrightarrow \mathcal Y
$$
be a birational isomorphism of projective orbifolds over $k$. 
Let $\mathcal{E}x(\phi)$ be 
the set of divisorial components of the exceptional locus. 
In parallel, we have

\begin{theo}
\label{thm:main2}
The assignment
\begin{equation*}
\label{eqn:cgg}
\mathcal C(\phi) := \sum_{\substack{\mathcal E\in \mathcal Ex(\phi^{-1})\\ \mathrm{gen. stab}(\mathcal E)=\{1\}}}
[\mathcal E] 
-\sum_{\substack{\mathcal D\in \mathcal Ex(\phi)\\ \mathrm{gen. stab}(\mathcal D)=\{1\}}} 
[\mathcal D] \in \bZ[\cBir_{n-1}],
\end{equation*}
where the sums are over divisorial components of exceptional loci with trivial generic stabilizer, 
respects compositions 
$$
\mathcal C(\psi\circ\phi) = \mathcal C (\phi)+\mathcal C(\psi), 
$$
for sequences of birational isomorphisms of orbifolds
$$
\cX\stackrel{\phi}{\dashrightarrow} \cY\stackrel{\psi}{\dashrightarrow} \cZ.
$$
\end{theo}

\medskip
\noindent
\textbf{Acknowledgments:}
The first author was partially supported by the
Swiss National Science Foundation. 
The second author was partially supported by NSF grant 2000099.

\section{Grothendieck ring and (stable) birational geometry}

Let $k$ be an algebraically closed field of characteristic zero. Let 
$$
\Var_n=\Var_{n,k}
$$
be the set of isomorphism classes of algebraic varieties over $k$ of dimension $n$, and put
$$
\Var = \bigsqcup_{n\ge 0} \Var_{n}.
$$
The {\em Grothendieck ring}
$$
\mathrm K_0(\Var)
$$
is the quotient of the $\bZ$-module 
$$
\bZ[\Var],
$$
the free abelian group spanned by isomorphism classes of algebraic varieties over $k$, 
by the  {\em Excision} and {\em Product} relations;
it carries a natural filtration, by dimension. We will write $[X]$ for the class of a variety in $\mathrm K_0(\Var)$.

The paper \cite{larsenlunts} revealed a connection with birational geometry: smooth projective varieties $X$ and $Y$ are {\em stably birational} if and only if 
$
[X] = [Y]
$
modulo $(\mathbb L)$, where $\mathbb L$ is the class of the affine line, 
so that 
$$
\mathrm K_0(\Var) / (\mathbb L) \cong \bZ[\mathrm{SBir}],
$$
the free abelian group spanned by {\em stable birationality classes} of algebraic varieties over $k$. 

The paper \cite{KT} introduced 
$$
\Burn_{n}=\bZ[\Bir_{n}],
$$
the free abelian group spanned by {\em birationality} classes of  $n$-dimensional algebraic varieties over $k$. 
One has natural isomorphisms
$$
\mathrm K_0(\Var^{\le n})/\mathrm K_0(\Var^{\le n-1}) \stackrel{\sim}{\lra} \Burn_n, 
$$
from {\em truncated}, by dimension, $\mathrm K_0$-groups. The group 
$$
\Burn:=\bigoplus_{n\ge 0} \Burn_{n}
$$
has a natural ring structure, induced by the product operation on algebraic varieties.  
There is also a surjection
$$
\Burn \to \mathrm{gr}(\mathrm K_0(\Var)),
$$
with nontrivial kernel  \cite[Theorem 2.13]{borisov}.


\section{Generalites: equivariant geometry}
\label{sec.generalities}
We continue to work over an algebraically closed field $k$ of characteristic zero.
Varieties over $k$ are irreducible and reduced.
For varieties with an action of a finite group $G$ we follow the conventions of \cite[\S 2]{KT-vector}: there might be several irreducible components, but the $G$-action is assumed to be transitive on irreducible components.
Usually the $G$-action is generically free, meaning that the locus with trivial stabilizer is nonempty, and hence dense.
We write
\[ X \sim_G Y \]
to mean that two varieties $X$ and $Y$ with generically free $G$-actions are equivariantly birationally equivalent, i.e., there exists a $G$-equivariant birational isomorphism 
$$
\phi\colon X\dashrightarrow Y.
$$
The set of such $\phi$ will be denoted by 
$$
\Bir_G(X,Y).
$$
We put 
$$
\Bir\Aut_G(X) :=\Bir_G(X,X),
$$
and note that when $X$ is a rational variety of dimension $n$,
\begin{equation}
\label{eqn.AutG}
\Bir\Aut_G(X) \subseteq \Cr_n
\end{equation}
is the {\em centralizer} of $G$ in the Cremona group.

A generically free $G$-action on a smooth projective variety $X$ is called {\em linearizable}, respectively {\em stably linearizable}, if there exists a faithful $G$-representation $V$ such that 
$$
X\sim_G \bP(V), \quad \text{resp.} \quad X\times \bP^m \sim_G\bP(V)\quad \text{for some $m$}.
$$
(The $G$-action on $X\times \bP^m$
is the product of the given action on $G$ and the trivial action on $\bP^m$.)
The determination of equivariant birationality classes, and in particular, of linearizable actions, is a current research direction.

The equivariant Burnside group $\Burn_n(G)$ from \cite{BnG} was defined so that a smooth projective variety $X$ of dimension $n$ with generically free $G$-action determines a class
\[ [X\actsfromright G]\in \Burn_n(G), \]
in such a manner that
\[ X\sim_G X'\quad\Rightarrow\quad
[X\actsfromright G]=[X'\actsfromright G]\quad\text{in $\Burn_n(G)$}. \]
The abelian group $\Burn_n(G)$ is defined by generators and relations.
Generators are \emph{symbols}
\[ (H, Y\actsfromleft K, \beta) \]
consisting of an abelian subgroup $H\subseteq G$, a function field with faithful action of a subgroup $Y\subseteq Z_G(H)/H$, where $Z_G(H)$ denotes the centralizer of $H$ in $G$, and a finite sequence of nontrivial characters of $H$, generating $H^\vee$;
the transcendence degree of $K/k$ and the length of $\beta$ must sum to $n$.
Two symbols are declared equivalent, if they differ by conjugation by an element of $G$, or reordering of the characters in $\beta$.
Two further relations are imposed:
\begin{itemize}
\item
\emph{vanishing} of symbols with $\beta=(a,b,\dots)$, $a+b=0$;
\item
\emph{blow-up}, equating a symbol $(H,Y\actsfromleft K,\beta)$, $\beta=(a,b,\dots)$, with $\Theta_1+\Theta_2$, where
\begin{gather*}
\Theta_1=(H,Y\actsfromleft K,\beta_1)+(H,Y\actsfromleft K,\beta_2), \\
\beta_1=(a,b-a,\dots),\quad \beta_2=(b,a-b,\dots),
\end{gather*}
when $a$ and $b$ are distinct characters, $\Theta_1=0$ when $a=b$, and
\[
\Theta_2=(\overline{H},\overline{Y}\actsfromleft \overline{K},\overline\beta)
\]
when $b-a$ has nontrivial kernel $\overline{H}$ and the characters in $\beta$ remain nontrivial in $\overline{H}^\vee$, otherwise $\Theta_2=0$.
Here, $\overline{K}$ is $K(t)$ with a particular action lifting the action on $K$ of the group $\overline{Y}\subseteq Z_G(H)/\overline{H}$, containing $H/\overline{H}$, with $\overline{Y}/(H/\overline{H})=Y$.
The images in $\overline{H}^\vee$ of the characters $b$, $\dots$ of $\beta$ form $\bar\beta$.
\end{itemize}

We have omitted several details, including a nontrivial condition on $Y\actsfromleft K$, for $(H,Y\actsfromleft K,\beta)$ to be admitted as a symbol.
This condition, called \emph{Assumption 1} in \cite{BnG}, ensures that characters of $H$ can be lifted to certain $1$-cocycles, used in the \emph{action construction} of \cite{BnG}, which gives the action on $\overline{K}$ in the term $\Theta_2$ of the blow-up relation.

Given a smooth projective variety $X$ with generically free $G$-action, we first perform a particular sequence of blow-ups in $G$-invariant centers to obtain an action in \emph{standard form}: there is open invariant $U\subset X$, on which $G$-acts freely, such that $X\setminus U$ is a simple normal crossing divisor such that the $G$-orbit of every component is nonsingular.
An action in standard form has abelian stabilizers.
Making a choice of orbit representatives $F$ of the finite $G$-set
\[ \text{\{components of $X^H$}\,|\,\text{$H\subseteq G$ abelian}\}, \]
we define
\[
[X\actsfromright G]:=\sum_{\text{orbit rep.\ $F$}}
(H,Y\actsfromleft k(F),\beta),
\]
where $H$ is the generic stabilizer of $F$, the quotient by $H$ of the subgroup of $G$, mapping $F$ to itself, is taken as $Y$ with induced action on the function field $k(F)$, and $\beta$ records the representation type of $H$ on the normal bundle to $F$ in $X$ at a general point of $F$.

\begin{theo}[\cite{BnG}]
\label{thm.BnG}
The class
\[ [X\actsfromright G]\in \Burn_n(G) \]
is an invariant of the $G$-equivariant birational type of an $n$-dimensional smooth projective variety $X$ with generically free $G$-action.
\end{theo}

Since the vanishing and blow-up relations involve only symbols with \emph{nontrivial} abelian subgroups of $G$, we obtain a direct sum decomposition
$$
\Burn_n(G)=\bZ[\Bir_{G,n}] \oplus \Burn_n^{\mathrm{nontriv}}(G).
$$
The first factor consists of symbols $(1,Y\actsfromleft K,())$ up to conjugation; these encode the $G$-equivariant birational type of a smooth projective variety of dimension $n$ with generically free $G$-action.
The second factor is generated by symbols with nontrivial $H$. 
In applications we consider the projection of $[X\actsfromright G]$ to the second factor and exploit features which potentially distinguish $G$-equivariant birational types.

\section{Equivariant classes of divisors}
\label{sec.divisors}
Here we prove Theorem \ref{thm:main}.
As recalled in Section \ref{sec.generalities}, a smooth projective variety $X$ with generically free $G$-action determines a class
\[ [X\actsfromright G]\in \Burn_n(G). \]
This definition may be extended to possibly singular projective varieties with generically free $G$-action, by first applying equivariant resolution of singularities.
Theorem \ref{thm.BnG} extends to possibly singular projective varieties with generically free $G$-action.

A $G$-equivariant birational map
\[ \phi\colon X\dashrightarrow Y \]
of smooth projective varieties with generically free $G$-action fits into a diagram of $G$-equivariant birational morphisms of smooth projective varieties
\begin{equation}
\begin{split}
\label{eqn.biratfactor}
\xymatrix@C=12pt{
& Z\ar[dl]_{\sigma}\ar[dr]^{\tau} \\
X \ar@{-->}[rr]_{\phi} && Y
}
\end{split}
\end{equation}
(We may apply equivariant resolution of singularities to the closure in $X\times Y$ of the graph of $\phi$ to obtain such a diagram.)

\begin{lemm}
\label{lem.step1}
In the situation of diagram \eqref{eqn.biratfactor} and under the convention that we identify an orbit of divisors on $X$ or on $Y$ with the orbit of proper transforms in $Z$ we have:
\begin{align*}
\mathrm{Ex}(\sigma^{-1})&=\emptyset, \\
\mathrm{Ex}(\tau^{-1})&=\emptyset, \\
\mathrm{Ex}(\phi)&= \mathrm{Ex}(\tau)\setminus\mathrm{Ex}(\sigma), \\
\mathrm{Ex}(\phi^{-1})&= \mathrm{Ex}(\sigma)\setminus\mathrm{Ex}(\tau). \\
\end{align*}
\end{lemm}

\begin{proof}
Since the exceptional locus of the inverse of a birational morphism has no divisorial components, the first two equalities are clear.
For any divisor in the exceptional locus of $\phi$, the proper transform in $Z$ is in the exceptional locus of $\tau$ but not of $\sigma$.
A divisor in the exceptional locus of $\tau$, not in the exceptional locus of $\sigma$ has image a divisor in $X$, in the exceptional locus of $\sigma$.
This establishes the third equality, and the fourth equality is established similarly.
\end{proof}

\begin{proof}[Proof of Theorem \ref{thm:main}]
Using Lemma \ref{lem.step1} and
the $G$-equivariant birational invariance of the class of a $G$-invariant divisor in $\Burn_{n-1}(G)$,
we obtain
\begin{equation}
\label{eqn.phisigmatau}
C_G(\phi)=-C_G(\sigma)+C_G(\tau).
\end{equation}

By repeating the construction leading to \eqref{eqn.biratfactor} (for $\phi$, for $\psi$, and for the equivariant birational map from the variety dominating $X$ and $Y$ to the variety dominating $Y$ and $Z$) we obtain an analogous diagram
\[
\xymatrix{
& W\ar[dl]_{\sigma}\ar[d]^{\tau}\ar[dr]^{\omega} \\
X \ar@{-->}[r]_{\phi} & Y \ar@{-->}[r]_{\psi} & Z
}
\]
By applying \eqref{eqn.phisigmatau} to $\phi$, $\psi$, and $\psi\circ\phi$, we obtain the theorem.
\end{proof}

\begin{rema}
In the definition of $C_G(\phi)$, we consider only divisorial components of exceptional divisors with {\em trivial} generic stabilizer. 
One reason for this is clear: the definition of $[X\actsfromright G]$ in \cite{BnG} requires a generically free $G$-action.
A formalism was developed in \cite{KT-struct} to encode actions with nontrivial generic stabilizer,
and one could imagine a formulation of $C_G(\phi)$ that would include divisorial orbits of components of exceptional loci with nontrivial generic stabilizer.
However, these contributions would give a quantity that does not depend on $\phi$; in fact, the quantity would be a difference of two quantities, one depending only on the $G$-equivariant isomorphism class of $X$, the other only on the $G$-equivariant isomorphism class of $Y$.
Indeed, there are only finitely many divisors with nontrivial generic stabilizer on $X$, and on $Y$.
In the situation of diagram \eqref{eqn.biratfactor} there are only finitely many divisors with nontrivial generic stabilizers on $Z$, containing (birationally modified copies of) the ones of $X$, and the ones of $Y$; application of \eqref{eqn.phisigmatau} leads to the observation.
\end{rema}

\section{Generalities: Orbifolds} 
\label{sec.orbi}
Another setting for the discussion of birational geometry is that of \emph{orbifolds}.
The main example of interest will be the quotient stack $[X/G]$ associated to a smooth projective variety $X$ with a generically free action of a finite group $G$.
Abstractly, an (algebraic) orbifold over $k$ is a smooth separated irreducible Deligne-Mumford stack of finite type over $k$, with trivial generic stabilizer.
For example, the orbifold $[X/G]$ is the category (a stack is a category with functor to a suitable base category, which for us will be $k$-schemes)
where an object over a $k$-scheme $T$ consists of a $G$-torsor $E\to T$ with a $G$-equivariant morphism $E\to X$
(where the functor to the base category sends such an object to the $k$-scheme $T$).
A morphism over $g\colon T\to T'$, from $E\to T$, with $E\to X$, to $E'\to T'$, with $E'\to X$, is a $G$-equivariant morphism $E\to E'$, that fits into a commutative diagram with $g$ and fits into a commutative diagram with the morphisms to $X$.

To every orbifold (or, more generally, Deligne-Mumford stack of finite type over $k$ with finite inertia, the last a technical condition implied by separatedness) there is an associated coarse moduli space; the coarse moduli space of $[X/G]$ is the quotient variety $X/G$.
When the coarse moduli space is a projective variety, we call the orbifold \emph{projective}; the notion of quasi-projective orbifold is defined analogously.
For instance, the orbifolds $[X/G]$ mentioned above are projective.

In \cite{Bbar} we studied invariants of projective orbifolds up to birational equivalence.
However, the notion of birational equivalence is subtle and encodes more information than the birational equivalence class of the coarse moduli space.

\begin{defi}[\cite{Bbar}]
\label{def.biratorbifolds}
Let $\mathcal{X}$ and $\mathcal{Y}$ be projective orbifolds over $k$.
We say that  $\mathcal{X}$ and $\mathcal{Y}$ are \emph{birationally equivalent} and employ the notation
\[ \mathcal{X}\sim\mathcal{Y}, \]
if there exists a commutative diagram of projective orbifolds
\begin{equation}
\begin{split}
\label{eqn.biratorbi}
\xymatrix@C=12pt{
& \mathcal{Z}\ar[dl]_{\sigma}\ar[dr]^{\tau} \\
\mathcal{X} \ar@{-->}[rr]_{\phi} && \mathcal{Y}
}
\end{split}
\end{equation}
where the morphisms from $\mathcal{Z}$ are projective and birational.
\end{defi}

In the setting of a diagram \eqref{eqn.biratorbi} we call $\phi\colon \mathcal{X}\dashrightarrow \mathcal{Y}$ a \emph{birational isomorphism},  write
\[ \Bir(\mathcal{X},\mathcal{Y}) \]
for the set of such, and put
\[ \Bir\Aut(\mathcal{X}):=\Bir(\mathcal{X},\mathcal{X}). \]

\begin{rema}
\label{rem.biratorbi}
Without insisting on a diagram \eqref{eqn.biratorbi}, to give
$\mathcal{X}\dashrightarrow\mathcal{Y}$ is the same as giving a birational isomorphism of the coarse moduli spaces.
The requirement of a diagram \eqref{eqn.biratorbi} makes the notion of birational isomorphism of orbifolds more subtle.
In contrast with the characterization \eqref{eqn.AutG} of $G$-equivariant birational automorphisms of rational varieties, there is no evident characterization of $\Bir\Aut(\mathcal{X})$ as a subgroup of the $k$-automorphisms of the field of rational functions $k(\mathcal{X})$.
\end{rema}

A composition of birational isomorphisms of projective orbifolds is a birational isomorphism. Indeed, given birational isomorphisms
\[
\mathcal{X}\stackrel{\phi}\dashrightarrow\mathcal{Y}\qquad\text{and}\qquad
\mathcal{Y}\stackrel{\psi}\dashrightarrow\mathcal{Z}
\]
we consider accompanying diagrams, in the manner of \eqref{eqn.biratorbi}, and with a construction of closure of graph and resolution of singularities obtain a commutative diagram of projective orbifolds
\begin{equation}
\begin{split}
\label{eqn.threeorbi}
\xymatrix{
& \mathcal{W}\ar[dl]_{\sigma}\ar[d]^{\tau}\ar[dr]^{\omega} \\
\mathcal{X} \ar@{-->}[r]_{\phi} & \mathcal{Y} \ar@{-->}[r]_{\psi} & \mathcal{Z}
}
\end{split}
\end{equation}
where the morphisms from $\mathcal{W}$ are projective and birational.

A birational isomorphism between orbifolds of dimension one (orbifold curves) is necessarily an isomorphism.
Many orbifold curves have the same coarse moduli space (e.g., quotients $[\bP^1/C_n]$ by cyclic groups of varying order $n$) but represent different isomorphism classes and hence different birational equivalence classes of orbifolds.

Here is an example of a nontrivial birational automorphism of an orbifold surface:

\begin{exam}
\label{exam:dp6stack}
 Consider 
$\bP^2$, with a linear action of a cyclic group of prime order $G:=C_p$, given by
$$
(x:y:z)\mapsto (x:\zeta^a_p y: \zeta^b_pz), \quad a,b \neq 0 \pmod p, \quad a\neq b \pmod p,
$$
and $\mathcal{X}:=[\bP^2/G]$.
We record the weights at the three fixed points:
$$
(a,b), \quad 
(-a,b-a), \quad 
(-b,a-b)
$$
We blow up the fixed points, to obtain $Y$, a del Pezzo surface of degree 6, and blow down the proper transforms of the coordinate lines; these are $G$-equivariant morphisms. Computing the weights in the resulting projective plane, we obtain
$$
(-a,-b), \quad 
(a,a-b), \quad 
(b,b-a). \quad 
$$
The automorphism of $G$ given by $c\mapsto -c$ reproduces the original weights and indeed the original quotient stack.
We obtain a diagram
\[
\xymatrix@C=12pt{
& [Y/G]\ar[dl]\ar[dr] \\
\mathcal{X}\ar@{-->}[rr]_{\phi} && \mathcal{X},
}
\]
with $\phi\in \Bir\Aut(\mathcal{X})$, not an automorphism of $\mathcal{X}$.
\end{exam}

The orbifold Burnside group $\overline{\Burn}_n$ was introduced in \cite{Bbar}.
An $n$-dimensional projective orbifold $\mathcal{X}$ determines a class
\[ [\mathcal{X}]\in\overline{\Burn}_n \]
with
\[ \mathcal{X}\sim \mathcal{Y}\quad\Rightarrow\quad
[\mathcal{X}]=[\mathcal{Y}]\quad\text{in $\overline{\Burn}_n$}. \]

\section{Classes of orbifold divisors}
\label{sec.orbifolddivisors}
Here we prove Theorem \ref{thm:main2}.
We start by noting that the definition
of the class of an $n$-dimensional projective orbifold
\[ [\mathcal{X}]\in\bZ[\cBir_{n-1,k}] \]
may be extended,
using resolution of singularities of Deligne-Mumford stacks (of finite type over $k$), to the class of possibly singular Deligne-Mumford stacks
$\mathcal{X}$, which are irreducible, reduced, separated, with projective coarse moduli space (i.e., irreducible, reduced projective Deligne-Mumford stacks over $k$).

\begin{lemm}
\label{lem.orbi1}
In the situation of diagram \eqref{eqn.biratorbi} and under the convention that we identify divisors on $\mathcal{X}$ or on $\mathcal{Y}$ with their proper transforms in $\mathcal{Z}$ we have:
\begin{align*}
\mathcal{E}x(\sigma^{-1})&=\emptyset, \\
\mathcal{E}x(\tau^{-1})&=\emptyset, \\
\mathcal{E}x(\phi)&=\mathcal{E}x(\tau)\setminus\mathcal{E}x(\sigma),\\
\mathcal{E}x(\phi^{-1})&=\mathcal{E}x(\sigma)\setminus\mathcal{E}x(\tau).
\end{align*}
\end{lemm}

\begin{proof}
The proof is analogous to that of Lemma \ref{lem.step1}.
\end{proof}

\begin{proof}[Proof of Theorem \ref{thm:main2}]
By Lemma \ref{lem.orbi1}, we have
\[
\mathcal{C}(\phi)=-\mathcal{C}(\sigma)+\mathcal{C}(\tau).
\]
This equality, applied to the maps in diagram \eqref{eqn.threeorbi}
completes the proof.
\end{proof}

\section{Example}
\label{sect:exam}
Here we rework an example from \cite{LSh} and show that the invariant $C_G$ can take a nonzero value.

\begin{exam}
\label{exa.LSh}
We adapt the construction of \cite[Theorem 3.7]{LSh} to obtain, for a linear action of $G:=C_5$ on $\bP^3$, a $G$-equivariant birational automorphism $\phi$ with \[C_G(\phi)\ne 0. \]
Let $V$ be the regular representation of $G$ and
\[ X:=G(2,V)\subset {\textstyle \bP(\bigwedge^2V)} \]
the Grassmannian of $2$-dimensional subspaces of $V$.
For a general linear form $\ell\in \bigwedge^2V^\vee$ we have:
\begin{itemize}
\item The $G$-orbit of $\ell$ cuts out in $X$ a nonsingular curve $C$, of genus $1$ with $j$-invariant different from $0$ and $1728$.
\item The (unique up to scale) nontrivial $G$-invariant quadratic form on $\bigwedge^2V$ vanishing on $G(2,V)$ defines, together with the $G$-orbit of $\ell$, a nonsingular quadric threefold $Q$.
\end{itemize}
As explained in \cite[Proposition 3.6]{LSh},
the additional quadratic forms on $\bigwedge^2V$ vanishing on $G(2,V)$ define a rational map
\[ Q\dashrightarrow \bP^3, \]
given by blowing up $C$ in $Q$ and contracting a divisor to the Jacobian
\[ J^2(C)\subset \bP^3 \]
of degree $2$ line bundles on $C$.
Another such rational map is given by projection from a fixed point on $Q$, i.e., blowing up the fixed point and contracting a divisor to a conic.
These are equivariant birational isomorphisms, which we combine to obtain $\phi$,
and we use that $C$ and $J^2(C)$ are not equivariantly isomorphic to obtain
the nonvanishing of $C_G(\phi)$.
\end{exam}

\bibliographystyle{plain}
\bibliography{orbibir}

\end{document}